\documentclass{article}%
\usepackage{amsfonts}
\usepackage{amsmath}
\usepackage{amssymb}
\usepackage{graphicx}
\usepackage{color,float}%
\providecommand{\U}[1]{\protect\rule{.1in}{.1in}}
\newtheorem{theorem}{Theorem}

\newtheorem{definition}[theorem]{Definition}

\newtheorem{lemma}[theorem]{Lemma}

\newtheorem{proposition}[theorem]{Proposition}

\newenvironment{proof}[1][Proof]{\noindent\textbf{#1.} }{\ \rule{0.5em}{0.5em}}
\begin{document}

\title{The sharp maximal function approach to \\$L^{p}$ estimates for operators structured \\on H\"{o}rmander's vector fields}
\author{Marco Bramanti, Marisa Toschi}
\maketitle

\begin{abstract}
We consider a nonvariational degenerate elliptic operator of the kind%
\[
Lu\equiv\sum_{i,j=1}^{q}a_{ij}(x)X_{i}X_{j}u
\]
where $X_{1},...,X_{q}$ are a system of left invariant, $1$-homogeneous,
H\"{o}rmander's vector fields on a Carnot group in $\mathbb{R}^{n}$, the
matrix $\left\{  a_{ij}\right\}  $ is symmetric, uniformly positive on a
bounded domain $\Omega\subset\mathbb{R}^{n}$ and the coefficients satisfy%
\[
a_{ij}\in VMO_{loc}\left(  \Omega\right)  \cap L^{\infty}\left(
\Omega\right)  .
\]
We give a new proof of the interior $W_{X}^{2,p}$ estimates%
\[
\left\Vert X_{i}X_{j}u\right\Vert _{L^{p}\left(  \Omega^{\prime}\right)
}+\left\Vert X_{i}u\right\Vert _{L^{p}\left(  \Omega^{\prime}\right)  }\leq
c\left\{  \left\Vert Lu\right\Vert _{L^{p}\left(  \Omega\right)  }+\left\Vert
u\right\Vert _{L^{p}\left(  \Omega\right)  }\right\}
\]
for $i,j=1,2,...,q$, $u\in W_{X}^{2,p}\left(  \Omega\right)  ,$ $\Omega
^{\prime}\Subset\Omega$ and $p\in\left(  1,\infty\right)  $, first proved by
Bramanti-Brandolini in \cite{bb1}, extending to this context Krylov'
technique, introduced in \cite{K1}, consisting in estimating the sharp maximal
function of $X_{i}X_{j}u.$

\end{abstract}

\section{Introduction}

Let us consider a linear second order elliptic operator in nondivergence form:%
\[
Lu\equiv\sum_{i,j=1}^{n}a_{ij}\left(  x\right)  u_{x_{i}x_{j}}%
\]
with $\left\{  a_{ij} \right\}  $ symmetric matrix of bounded measurable
functions defined on some domain $\Omega\subset\mathbb{R}^{n}$ and satisfying
the uniform ellipticity condition%
\[
\mu\left\vert \xi\right\vert ^{2}\leq\sum_{i,j=1}^{n}a_{ij}\left(  x\right)
\xi_{i}\xi_{j}\leq\frac{1}{\mu}\left\vert \xi\right\vert ^{2}%
\]
for some $\mu>0$, every $\xi\in\mathbb{R}^{n}$, a.e. $x\in\Omega$. While the
classical $W^{2,p}$-theory of elliptic equations, dating back to
Agmon-Douglis-Nirenberg \cite{ADN} and essentially exploiting the $L^{p}$
theory of singular integrals due to Calder\'{o}n-Zygmund \cite{CZ} requires
the uniform continuity of the coefficients $a_{ij}\left(  x\right)  $, in 1993
Chiarenza-Frasca-Longo \cite{Chiarenza-Frasca-Longo} proved $W^{2,p}$
estimates under the mere assumption $a_{ij}\in L^{\infty}\cap VMO$, which
allows for some kind of discontinuities in the coefficients. Their technique
is based on representation formulas of $u_{x_{i}x_{j}}$ by means of singular
integrals with variable kernels, and commutators of these singular integrals
with $BMO$ functions. Thanks to a deep real analysis theorem by
Coifman-Rochberg-Weiss \cite{CRW}, these commutators have small operator norm
on small balls, hence the old idea of seeing a variable coefficient operator
as a small perturbation of the model operator with constant coefficients is
ingeniously generalized to an operator with possibly discontinuous
coefficients. This technique, by now classic, has been extended to several
contexts, for instance parabolic operators (see \cite{Bramanti-Cerutti}) and
nonvariational operators structured on H\"{o}rmander's vector fields (see
\cite{bb1}, \cite{BB2}).

In 2007 Krylov \cite{K1} introduced a differerent technique to prove similar
and more general results for elliptic and parabolic operators, based on the
\emph{pointwise} estimate of the \emph{sharp maximal function of }%
$u_{x_{i}x_{j}}$, that is $\left(  u_{x_{i}x_{j}}\right)  ^{\#}$. The idea is
then again that of approximating the operator with variable coefficients with
a model operator with constant coefficients; these constants in this case are
not simply the original coefficients \emph{frozen }at some point, but suitable
integral averages of these functions. The theory of singular integrals is not
\emph{explicitly} used, but it is replaced by Fefferman-Stein maximal theorem,
which allows to control the $L^{p}$ norm of $u_{x_{i}x_{j}}$ by that of
$\left(  u_{x_{i}x_{j}}\right)  ^{\#}$. On the other hand, throughout the
computation which is carried out on the model operator, many classical results
are employed, implicitly involving also the classical Calder\'{o}n-Zygmund theory.

The research started with this paper aims to investigate whether Krylov'
technique can be extended also to the context of linear degenerate equations
structured on H\"{o}rmander's vector fields, and if it can be used to get new
results not easily obtainable with the techniques previously used. We give a
partial positive answer to this question for the class of operators%
\[
Lu\equiv\sum_{i,j=1}^{q}a_{ij}(x)X_{i}X_{j}u
\]
where $X_{1},...,X_{q}$ are a system of left invariant and $1$-homogeneous
H\"{o}rmander's vector fields on a Carnot group in $\mathbb{R}^{n}$, the
matrix $\left\{  a_{ij}\right\}  $ is symmetric, uniformly positive on a
bounded domain $\Omega\subset\mathbb{R}^{n}$ with $a_{ij}$ bounded measurable
and (locally in $\Omega$) $VMO$, with respect to the balls induced by the
vector fields. The assumption of existence of an underlying Carnot group
structure such that $L$ is translation invariant and $2$-homogeneous is quite
natural in consideration of the important role of dilations in Krylov'
approach. In this context we prove a pointwise bound on the local sharp
maximal function of $X_{i}X_{j}u.$ This, combined with an extension of
Fefferman-Stein's theorem to the context of locally homogeneous spaces,
recently obtained, (see \cite{BF}) allows to get the local estimates first
proved by Bramanti-Brandolini in \cite{bb1} with an approach that parallels
that of Chiarenza-Frasca-Longo.

The result, therefore, is not original; the novelty lies in the approach,
which allows some simplification with respect to that of \cite{bb1}. We hope
to extend in the future the present approach to different classes of
degenerate operators, getting some kind of new $L^{p}$ estimate.

\bigskip

\noindent\textbf{Acknowledgement. }This research was carried out while Marisa
Toschi was visiting the Department of Mathematics of Politecnico di Milano,
which we want to thank for hospitality. This author was partially supported by
Universidad Nacional del Litoral through grants CAI+D 50020110100009.

\section{Preliminaries and known results}

\subsection{Carnot groups}

We start recalling some standard terminology and known facts about Carnot
groups. For more details and for the proofs of known results the reader is
referred to \cite[Chaps. 1, 2]{BLUbook}, \cite{fo}, \cite[Chap.XIII,
\S 5]{Stein}.

We call \emph{homogeneous group }the space $\mathbb{R}^{n}$ equipped with a
Lie group structure, together with a family of dilations that are group
automorphisms. Explicitly, assume that we are given a pair of mappings:
\[
\lbrack(x,y)\mapsto x\circ y]:\mathbb{R}^{n}\times\mathbb{R}^{n}%
\rightarrow\mathbb{R}^{n}\text{ and }\left[  x\mapsto x^{-1}\right]
:\mathbb{R}^{n}\rightarrow\mathbb{R}^{n}%
\]
that are smooth and such that $\mathbb{R}^{n}$, together with these mappings,
forms a group, for which the identity is the origin. We will think to the
operation $\circ$ as a \emph{translation}. Next, suppose that we are given an
$n$-tuple of strictly positive integers $\alpha_{1}\leq\alpha_{2}\leq
...\leq\alpha_{n}$, such that the \emph{dilations}%
\begin{equation}
D(\lambda):(x_{1},...,x_{n})\rightarrow(\lambda^{\alpha_{1}}x_{1}%
,...,\lambda^{\alpha_{n}}x_{n}) \label{dilations}%
\end{equation}
are group automorphisms, for all $\lambda>0$. We will denote by $\mathbb{G}$
the space $\mathbb{R}^{n}$ with this structure of homogeneous group, and we
will say that a constant depend on $\mathbb{G}$ if it depends on the numbers
$n$, $\alpha_{1},...,\alpha_{n}$ and the group law $\circ$.

We say that a differential operator $Y$ on $\mathbb{G\ }$is \emph{homogeneous
of degree }$\beta>0$ if
\[
Y\left(  f\left(  D(\lambda)x\right)  \right)  =\lambda^{\beta}(Yf)(D(\lambda
)(x))
\]
for every test function $f,\lambda>0$, $x\in\mathbb{R}^{n}$. Also, we say that
a function $f$ is homogeneous of degree $\alpha\in\mathbb{R}$ if
\[
f\left(  D(\lambda)x\right)  =\lambda^{\alpha}f(x)\text{ for every }%
\lambda>0,x\in\mathbb{R}^{n}\setminus\left\{  0\right\}  .
\]
Clearly, if $Y$ is a differential operator homogeneous of degree $\beta$ and
$f$ is a homogeneous function of degree $\alpha$, then $Yf$ is homogeneous of
degree $\alpha-\beta$.

We say that a differential operator $Y$ on $\mathbb{G}$ is \emph{left
invariant }if for every smooth function $f:\mathbb{G}\rightarrow\mathbb{R},$%
\begin{align*}
Y\left(  f\left(  L_{y}\left(  x\right)  \right)  \right)   &  =\left(
Yf\right)  \left(  y\circ x\right)  \text{ for every }x,y\in\mathbb{G},\\
\text{where }L_{y}\left(  x\right)   &  =y\circ x.
\end{align*}

Let us now consider the Lie algebra $\ell$ associated to the group
$\mathbb{G}$, that is, the Lie algebra of left invariant vector fields on
$\mathbb{G}$, endowed with the Lie bracket given by the commutator of vector
fields: $\left[  X,Y\right]  =XY-YX$. We can fix a basis $X_{1},...,X_{N}$ in
$\ell$ choosing $X_{i}$ as the (unique) left invariant vector field which
agrees with $\frac{\partial}{\partial x_{i}}$ at the origin. It turns out that
$X_{i}$ is homogeneous of degree $\alpha_{i}$. Then, we can extend the
dilations $D(\lambda)$ to $\ell$ setting
\[
D(\lambda)X_{i}=\lambda^{\alpha_{i}}X_{i}.
\]
$D\left(  \lambda\right)  $ becomes a Lie algebra automorphism, i.e.,
\[
D(\lambda)[X,Y]=[D(\lambda)X,D(\lambda)Y].
\]
In this sense, $\ell$ is said to be a \emph{homogeneous Lie algebra}; as a
consequence, $\ell$ is nilpotent.

We will assume that the first $q$ vector fields $X_{1},...,X_{q}$ are
$1$-homogeneous and generate $\ell$ as a Lie algebra. In other words,
$X_{1},...,X_{q}$ are a system of H\"{o}rmander's vector fields in
$\mathbb{R}^{n}$: there exists a positive integer $s$, called the step of the
Lie algebra, such that $X_{1},...,X_{q},$ together with their iterated
commutators of length $\leq s$ span $\mathbb{R}^{n}$ at every point. Under
these assumptions we say that $\ell$ is a stratified homogeneous Lie algebra
and that $\mathbb{G}$ is a stratified homogeneous group, or briefly a
\emph{Carnot group}.

As any system of H\"{o}rmander's vector fields, $X_{1},...,X_{q}$ induce in
$\mathbb{R}^{n}$ a distance $d$ called \emph{the control distance}. The
explicit definition of $d$ will never be used, hence we do not recall it (see
\cite{NSW}). Since $\mathbb{G}$ is a Carnot group, $d$ turns out to be left
invariant and $1$-homogeneous, that is%
\begin{align*}
d\left(  x,y\right)   &  =d\left(  z\circ x,z\circ y\right) \\
d\left(  D\left(  \lambda\right)  x,D\left(  \lambda\right)  y\right)   &
=\lambda d\left(  x,y\right)
\end{align*}
for any $x,y,z\in\mathbb{G}$ and $\lambda>0.$ Then, if we set%
\[
\Vert x\Vert=d\left(  x,0\right)  \text{,}%
\]
it turns out that $\Vert\cdot\Vert$ is a \emph{homogeneous norm}, satisfying
the following properties:

\begin{itemize}
\item[(i)] $\Vert D(\lambda)x\Vert=\lambda\Vert x\Vert$ for every
$x\in\mathbb{R}^{n},\lambda>0$;

\item[(ii)] the function $x\mapsto\Vert x\Vert$ is continuous;

\item[(iii)] for every $x,y\in\mathbb{R}^{n}$
\[
\left\Vert x\circ y\right\Vert \leq\left\Vert x\right\Vert +\left\Vert
y\right\Vert \text{ and }\left\Vert x^{-1}\right\Vert =\Vert x\Vert;
\]

\item[(iv)] there exists a constant $c\geq1$ such that%
\[
\frac{1}{c}\left\vert y\right\vert \leq\left\Vert y\right\Vert \leq
c\left\vert y\right\vert ^{1/s}\text{ if }\Vert y\Vert\leq1,
\]
where $s$ is the step of the Lie algebra.
\end{itemize}

Note that from (iii) we have that%
\begin{equation}
\Vert y^{-1}\circ x\Vert\geq\Vert y\Vert-\Vert x\Vert.
\label{inequality of the distance}%
\end{equation}

We also define the balls with respect to $d$ as
\[
B(x,r)\equiv B_{r}(x)\equiv\{y\in\mathbb{R}^{n}:d(x,y)<r\},
\]
and denote $B_{r}=B(0,r)$.

Note that $B(0,r)=D(r)B(0,1)$. It can be proved that the Lebesgue measure in
$\mathbb{R}^{n}$ is the Haar measure of $\mathbb{G}$ and%
\begin{equation}
\left\vert B(x,r)\right\vert =\left\vert B(0,1)\right\vert r^{Q},
\label{vol balls}%
\end{equation}
for every $x\in\mathbb{R}^{n}$ and $r>0$, where
\[
Q=\alpha_{1}+...+\alpha_{n}%
\]
with $\alpha_{i}$ as in (\ref{dilations}). We will call $Q$ the
\emph{homogeneous dimension }of $\mathbb{G}$.

\subsection{Real analysis tools}

We start noting that (\ref{vol balls}) in particular implies that the Lebesgue
measure $dx$ is a \emph{doubling measure }with respect to $d$, and therefore
$(\mathbb{R}^{n},d,dx)$ is a \emph{space of homogenous type }in the sense of
Coifman-Weiss (see \cite{Coifman-Weiss}).

In this context, for a given locally integrable function $f$, the
Hardy-Littlewood maximal operator is given by
\begin{equation}
Mf(x)=\sup_{B\ni x}\frac{1}{|B|}\int_{B}|f(y)|dy, \label{maximal}%
\end{equation}
where the supremum is taken over all the $d$-balls (containing the point $x$).
By the general theory of spaces of homogeneous type, it is known that for
every $p\in\left(  1,\infty\right)  $ there exists a constant $c>0$ such that%
\begin{equation}
\Vert Mf\Vert_{L^{p}(\mathbb{R}^{n})}\leq c \Vert f\Vert_{L^{p}(\mathbb{R}%
^{n})}. \label{boundedness of the maximal}%
\end{equation}

Since we will study a differential operator defined on a bounded domain
$\Omega\subset\mathbb{R}^{n}$ and we will prove interior estimates in $\Omega
$, a natural framework for the real analysis tools we need is that of
\emph{locally homogeneous spaces}, as developed in \cite{BZ} and \cite{BF}. We
are going to introduce the minimum amount of definitions in order to apply
this abstract theory in our concrete context. So, for a fixed bounded domain
$\Omega\subset\mathbb{R}^{n}$, fix a strictly increasing sequence
$\{\Omega_{m}\}_{m=1}^{\infty}$ of bounded domains such that
\[
\bigcup_{m=1}^{\infty}\Omega_{m}=\Omega
\]
and such that for any $m$ there exists $\varepsilon_{m}>0$ such that
\[
\{x\in\Omega:d(x,y)<2\varepsilon_{m}\text{ for some }y\in\Omega_{m}%
\}\subset\Omega_{m+1}%
\]
where $d$ is, as above, the distance induced in $\mathbb{R}^{n}$ by the vector
fields $X_{i}$. Then $\left(  \Omega,\{\Omega_{m}\}_{m=1}^{\infty
},d,dx\right)  $ (where $dx$ stands for the Lebesgue measure) is a locally
homogeneous space in the sense of \cite{BZ}.

With respect to this structure, we can define the \emph{local sharp maximal
operator}: for any function $f\in L_{loc}^{1}\left(  \Omega_{m+1}\right)  $
and $x\in\Omega_{m}$, let%
\begin{equation}
f_{\Omega_{m},\Omega_{m+1}}^{\#}\left(  x\right)  =\sup_{\substack{B\left(
\overline{x},r\right)  \ni x\\\overline{x}\in\Omega_{m},r\leq\varepsilon_{m}%
}}\frac{1}{\left\vert B\left(  \overline{x},r\right)  \right\vert }%
\int_{B\left(  \overline{x},r\right)  }\left\vert f\left(  y\right)
-f_{B\left(  \overline{x},r\right)  }\right\vert dy. \label{sharp maximal}%
\end{equation}

Note that the supremum is taken over all the $d$-balls containing the point
$x\in\Omega_{m}$ and having radius small enough so that the ball itself is
contained in the larger set $\Omega_{m+1}$ where the function $f$ is defined.
Thus, we focus on the behavior of $f$ on a bounded domain but on the other
hand avoid the necessity of integrating over \emph{restricted balls }$B\left(
\overline{x},r\right)  \cap\Omega_{m+1}$. The continuity of the sharp maximal
operator is contained in the next result:

\begin{theorem}
[{Local Fefferman-Stein inequality, see \cite[Corollary 3.9]{BF}}%
]\label{Thm Fefferman Stein}There exists $\delta\in(0,1)$ such that for any
$m$ and for every integer $k$ large enough, the set $\Omega_{m}$ can be
covered by a finite union of balls $B_{R}$ of radii comparable to $\delta^{k}%
$, such that for any such ball $B_{R}$ and every $f$ supported in $B_{R}$,
with $f\in L^{1}\left(  B_{R}\right)  $, $\int_{B_{R}}f=0$, and $f_{\Omega
_{m+2},\Omega_{m+3}}^{\#}\in L_{loc}^{p}\left(  \Omega_{m+1}\right)  $ for
some $p\in\lbrack1,\infty)$ one has%
\[
\left\Vert f\right\Vert _{L^{p}\left(  B_{R}\right)  }\leq c\left\Vert
f_{\Omega_{m+2},\Omega_{m+3}}^{\#}\right\Vert _{L^{p}\left(  B_{\gamma
R}\right)  }%
\]
with $\gamma>1$ absolute constant and $c$ only depending on $p$, the sets
$\Omega_{k}$ and the constants $\varepsilon_{k}$ for a finite number of
indices $k$.
\end{theorem}

Let us also define the \emph{local VMO spaces}.

For a fixed $\Omega_{m}$, $f\in L_{loc}^{1}\left(  \Omega_{m+1}\right)  $ and
$0<r\leq\varepsilon_{m}$, let%
\[
\eta_{m,f}\left(  r\right)  =\sup_{x\in\Omega_{m},\rho\leq r}\frac
{1}{\left\vert B\left(  x,\rho\right)  \right\vert }\int_{B\left(
x,\rho\right)  }\left\vert f\left(  y\right)  -f_{B\left(  x,\rho\right)
}\right\vert dy.
\]
We say that $f\in VMO_{loc}\left(  \Omega_{m},\Omega_{m+1}\right)  $ if
$\eta_{m,f}\left(  r\right)  \rightarrow0$ for $r\rightarrow0^{+}$.

We say that a function $f\in L_{loc}^{1}\left(  \Omega\right)  $ belongs to
$VMO_{loc}\left(  \Omega\right)  $ if%
\[
\eta_{f}\left(  r\right)  \equiv\sup_{x\in\Omega,\rho\leq r,B\left(
x,\rho\right)  \Subset\Omega}\frac{1}{\left\vert B\left(  x,\rho\right)
\right\vert }\int_{B\left(  x,\rho\right)  }\left\vert f\left(  y\right)
-f_{B\left(  x,\rho\right)  }\right\vert dy\rightarrow0\text{ for
}r\rightarrow0^{+}.
\]

Note that the requirement $B\left(  x,\rho\right)  \Subset\Omega$ is
meaningful because the distance $d$ is define in the whole $\mathbb{R}^{n}$,
not only in $\Omega$. Observe that%
\begin{equation}
VMO_{loc}\left(  \Omega\right)  \subset\bigcap\limits_{m=1}^{\infty}%
VMO_{loc}\left(  \Omega_{m},\Omega_{m+1}\right)  . \label{VMO inclusion}%
\end{equation}

\subsection{Sobolev spaces and fundamental solutions}

Let us introduce some useful notation. For $X_{1},...,X_{q}$ the vector fields
as above and any multiindex $I=\left(  i_{1},...,i_{k}\right)  $ with
$i_{j}\in\left\{  1,2,...,q\right\}  $ we set%
\[
X_{I}u=X_{i_{1}}X_{i_{2}}...X_{i_{k}}u,\text{ \ }\left\vert I\right\vert =k.
\]
We then define, for any positive integer $k$,%
\[
D^{k}u\equiv\sum_{\left\vert I\right\vert =k}\left\vert X_{I}u\right\vert .
\]
(We will write $Du$ instead of $D^{1}u$). Here the $X_{i}$-derivatives are
meant in classical or weak sense. For $\Omega$ a domain in $\mathbb{R}^{n}$
and $p\in\lbrack1,\infty]$ the space $W_{X}^{k,p}(\Omega)$ will consist of all
$L^{p}(\Omega)$ functions such that
\[
\Vert u\Vert_{W_{X}^{k,p}(\Omega)}=\sum_{h=0}^{k}\Vert D^{h}u\Vert
_{L^{p}(\Omega)}%
\]
is finite (with $\Vert D^{0}u\Vert_{L^{p}(\Omega)}=\Vert u\Vert_{L^{p}%
(\Omega)}$). We shall also denote by $W_{X,0}^{k,p}(\Omega)$ the closure of
$C_{0}^{\infty}(\Omega)$ in $W_{X}^{k,p}(\Omega)$. Note that the fields
$X_{i}$, and therefore the definition of the above norms and spaces, are
completely determined by the structure of $\mathbb{G}$.

A couple of standard facts about these Sobolev spaces on Carnot groups are the following:

\begin{theorem}
[Poincar\'{e}'s inequality on stratified groups, see \cite{Jer}]%
\label{poincare}Let $\mathbb{G}$ be a Carnot group with generators
$X_{1},...,X_{q}$. For every $p\in\lbrack1,\infty)$ there exist constants
$c>0,\Lambda>1$ such that for any ball $B=B\left(  x_{0},r\right)  $ and any
$u\in C^{1}\left(  \overline{\Lambda B}\right)  $ (with $\Lambda B=B\left(
x_{0},\Lambda r\right)  $) we have:%
\[
\left(  \frac{1}{\left\vert B\right\vert }\int_{B}\left\vert u\left(
x\right)  -u_{B}\right\vert ^{p}dx\right)  ^{1/p}\leq cr\left(  \frac
{1}{\left\vert \Lambda B\right\vert }\int_{\Lambda B}\left\vert Du\left(
x\right)  \right\vert ^{p}dx\right)  ^{1/p}.
\]

\end{theorem}

Note that the constants $c,\Lambda$ in the previous Poincar\'{e}'s inequality
are independent of $r$ and $x_{0}$.

\begin{proposition}
[{Interpolation inequality, see \cite[Prop. 4.1]{bb1}}]%
\label{interpolation inequality}Let $X$ be a left invariant vector field
homogeneous of degree $1$. Then for every $\varepsilon>0$ and $u\in
W_{X,0}^{2,p}(\mathbb{R}^{n})$ with $p\in\lbrack1,\infty)$,
\[
\Vert Xu\Vert_{L^{p}}\leq\varepsilon\Vert X^{2}u\Vert_{L^{p}}+\frac
{2}{\varepsilon}\Vert u\Vert_{L^{p}}.
\]

\end{proposition}

Let us now consider the class of \emph{model operators}%
\begin{equation}
\overline{L}u(x)=\sum_{i,j=1}^{q}\overline{a}_{ij}X_{i}X_{j}u(x)
\label{constant coefficient operator}%
\end{equation}
where the matrix $\{\overline{{a}}_{ij}\}$ is constant, symmetric and
satisfies the ellipticity condition: there exists $\mu>0$ such that
\begin{equation}
\mu|\xi|^{2}\leq\overline{a}_{ij}\xi_{i}\xi_{j}\leq\frac{1}{\mu}|\xi|^{2}
\label{ellipticity}%
\end{equation}
for every $\xi\in\mathbb{R}^{q}$.

The operator $\overline{L}$ is a left invariant differential operator
homogeneous of degree two on $\mathbb{G}$; it is easy to see that
$\overline{L}$ can be rewritten in the form $\overline{L}=\sum_{i=1}^{q}%
Y_{i}^{2}$ where $Y_{1},...,Y_{q}$ are a different system of H\"{o}rmander'
vector fields (for details, see \cite[\S 2.4]{bb1}); hence $\overline{L}$ is
hypoelliptic, by H\"{o}rmander's theorem (see \cite{H}). By general properties
of Carnot groups, the formal transposed of $X_{i}$ is $X_{i}^{\ast}=-X_{i}$;
hence the transposed of $\overline{L}$ is still $\overline{L}$; in particular,
both $\overline{L}$ and $\overline{L}^{\ast}$ are hypoelliptic. We can
therefore apply the theory developed by Folland \cite{fo} about the
fundamental solution of $\overline{L}$. The following theorem collects the
properties we will need:

\begin{theorem}
[Homogeneous fundamental solution of $\overline{L}$]\label{Thm Folland}The
operator $\overline{L}$ has a unique global fundamental solution
$\Gamma_{\overline{a}}\leq0$ with pole at the origin wich is homogeneous of
degree $2-Q$ and such that:

\begin{enumerate}
\item[(a)] $\Gamma_{\overline{a}}\in C^{\infty}(\mathbb{R}^{n}\setminus\{0\})$;

\item[(b)] for every $u\in C_{0}^{\infty}\left(  \mathbb{R}^{n}\right)  $ and
every $x\in\mathbb{R}^{n}$,
\[
u(x)=\overline{L}u\ast\Gamma_{\overline{a}}(x)=\int_{\mathbb{R}^{n}}%
\Gamma_{\overline{a}}(y^{-1}\circ x)\overline{L}u(y)dy;
\]

\item[(c)] for every $f\in L^{2}\left(  \mathbb{R}^{n}\right)  ,f$ compactly
supported, the function%
\[
u(x)=f\ast\Gamma_{\overline{a}}(x)=\int_{\mathbb{R}^{n}}\Gamma_{\overline{a}%
}(y^{-1}\circ x)f(y)dy
\]
belongs to $W_{X}^{2,2}\left(  \mathbb{R}^{n}\right)  $ and solves the
equation $\overline{L}u=f$ in $\mathbb{R}^{n}$.
\end{enumerate}
\end{theorem}

We also need some uniform bound for $\Gamma_{\overline{a}}$, with respect to
the constant matrix $\{\overline{a}_{ij}\}$ in a fixed ellipticity class. The
next result is contained in \cite[Thm. 12]{bb1}:

\begin{proposition}
[Uniform estimate on $\Gamma_{\overline{a}}$]%
\label{bounded not depending on the coeficients}There exists a positive
constant, depending on $\{\overline{a}_{ij}\}$ only through the number $\mu$,
such that
\[
\left\vert \Gamma_{\overline{a}}(x)\right\vert \leq\frac{c}{\left\Vert
x\right\Vert ^{Q-2}}\text{ for every }x\in\mathbb{R}^{n}\setminus\left\{
0\right\}  \text{.}%
\]

\end{proposition}

Another key tool that we need from the general theory of H\"{o}rmander's
operators is represented by the so-called subelliptic estimates. To formulate
these, we need to recall the standard definition of (Euclidean, isotropic)
fractional Sobolev spaces: for any $s\in\mathbb{R}$ the space $H^{s}$ is
defined as the set of functions (or tempered distributions) such that%
\[
\left\Vert u\right\Vert _{H^{s}}^{2}=\int_{\mathbb{R}^{n}}\left(  1+\left\vert
\xi\right\vert ^{2}\right)  ^{s}\left\vert \widehat{u}\left(  \xi\right)
\right\vert ^{2}d\xi
\]
is finite, where $\widehat{u}\left(  \xi\right)  $ denotes the Fourier
transform of $u$. Then:

\begin{theorem}
[Subelliptic estimates, see \cite{Ko}]\label{subelliptic estimate}There exists
$\varepsilon>0$, depending on the $X_{i}$ and, for every $\eta,\eta_{1}\in
C_{0}^{\infty}\left(  \mathbb{R}^{n}\right)  $ with $\eta_{1}=1$ on
$\operatorname*{sprt}\eta$ and any $\sigma,\tau>0$, there exists a constant
$c$ depending on $\sigma,\tau,\eta,\eta_{1},X_{i}$ such that%
\[
\left\Vert \eta u\right\Vert _{H^{\sigma+\varepsilon}}\leq c\left(  \left\Vert
\eta_{1}\overline{L}u\right\Vert _{H^{\sigma}}+\left\Vert \eta_{1}u\right\Vert
_{H^{-\tau}}\right)
\]
where $\overline{L}$ is like in (\ref{constant coefficient operator}).
Moreover, the constant $c$ depends on the coefficients $\overline{a}_{ij}$
only through the number $\mu$.
\end{theorem}

Classical subelliptic estimates are proved for a \emph{fixed }operator of
H\"{o}rmander's type; however, the last statement about the dependence of $c$
on the $\overline{a}_{ij}$ can be directly checked following the proof.

For the operator $\overline{L}$ we can give a standard definition of weak
solution to a Dirichlet problem:

\begin{definition}
Let $\Omega$ a bounded domain. Given two functions $f\in W_{X}^{1,2}%
(\Omega),g\in L^{2}\left(  \Omega\right)  $, we say that $u\in W_{X}%
^{1,2}(\Omega)$ is a weak solution to the Dirichlet problem%
\begin{equation}
\left\{
\begin{array}
[c]{ll}%
\overline{L}u=g & \text{in }\Omega\\
u=f & \text{on }\partial\Omega
\end{array}
\right.  \label{DP}%
\end{equation}
if $u-f\in W_{X,0}^{1,2}(\Omega)$ and%
\[
-\int_{\Omega}\sum_{i,j=1}^{q}\overline{a}_{ij}X_{j}uX_{i}\varphi=\int
_{\Omega}g\varphi\quad\forall\varphi\in C_{0}^{\infty}(\Omega).
\]

\end{definition}

The validity of Poincar\'{e}'s inequality allows to prove in the standard way,
by Lax-Milgram's Lemma, the \emph{unique solvability of }(\ref{DP}). We stress
the fact that, although the operator $\overline{L}$ is hypoelliptic, so that
any distributional solution to $\overline{L}u=g$ is smooth in any open subset
where $g$ is smooth, the solvability of a Dirichlet problem in classical sense
is not a trivial result for $\overline{L}$, but requires careful assumptions
on the domain. Also, $W_{X}^{2,p}\left(  \Omega\right)  $ estimates up to the
boundary are not known, so far, so that the Dirichlet problem is not even
solvable in the sense of strong solutions. This is a major difference between
the present context and that of elliptic and parabolic equations, in the
application of Krylov' technique.

A maximum principle for weak solutions can be easily proved in the standard
way. This requires some preliminary (standard) definition:

\begin{definition}
For $u\in W_{X}^{1,2}\left(  \Omega\right)  $, we say that%
\[
\overline{L}u\geq0\text{ in }\Omega
\]
in weak sense if%
\[
\int_{\Omega}\sum_{i,j=1}^{q}\overline{a}_{ij}X_{j}uX_{i}\varphi\leq
0\quad\forall\varphi\in C_{0}^{\infty}(\Omega),\varphi\geq0\text{ in }\Omega.
\]
We say that%
\[
u\leq0\text{ on }\partial\Omega
\]
in weak sense if%
\[
\max\left(  u,0\right)  \in W_{X,0}^{1,2}\left(  \Omega\right)  .
\]

\end{definition}

The following can be easily proved exactly like in the elliptic case:

\begin{proposition}
[Maximum Principle]\label{maximum principle}Let $\Omega$ an open set of
$\mathbb{R}^{n}$. For any $u\in W_{X}^{1,2}\left(  \Omega\right)  $, if
$\overline{L}u\geq0$ in $\Omega$ and $u\leq0$ on $\partial\Omega$ (in weak
sense), then $u\leq0$ in $\Omega$ a.e.
\end{proposition}

\subsection{Main result}

Let us now consider an operator%
\[
Lu\equiv\sum_{i,j=1}^{q}a_{ij}(x)X_{i}X_{j}u
\]
where $X_{1},...,X_{q}$, as above, are a system of left invariant and
$1$-homogeneous H\"{o}rmander's vector fields on a Carnot group in
$\mathbb{R}^{n}$, the matrix $\left\{  a_{ij}\right\}  $ is symmetric, the
coefficients satisfy
\begin{equation}
a_{ij}\in VMO_{loc}\left(  \Omega\right)  \cap L^{\infty}\left(
\Omega\right)  \label{H1}%
\end{equation}
on a bounded domain $\Omega\subset\mathbb{R}^{n}$, and the uniform positivity
condition holds: there exists $\mu>0$ such that
\begin{equation}
\mu|\xi|^{2}\leq{a}_{ij}(x)\xi_{i}\xi_{j}\leq\frac{1}{\mu}|\xi|^{2}
\label{uniform ellipticity}%
\end{equation}
for every $\xi\in\mathbb{R}^{q}$ and a.e. $x\in\Omega$.

By the assumption (\ref{H1}) and the inclusion (\ref{VMO inclusion}), if we
set%
\begin{equation}
a_{m,r}^{\sharp}=\sum_{i,j=1}^{q}\eta_{m,a_{ij}}\left(  r\right)  ,
\label{a sharp}%
\end{equation}
we have%
\[
\sup_{r\leq\varepsilon_{m}}a_{m,r}^{\sharp}<\infty\text{ and }\lim
_{r\rightarrow0^{+}}a_{m,r}^{\sharp}=0.
\]

The main result that can be proved is then the following:

\begin{theorem}
\label{Final result}Under the previous assumptions, for any $\Omega_{m}%
\Subset\Omega$ and $p\in\left(  1,\infty\right)  $ there exists a constant $c$
depending on $\Omega,\Omega_{m},p,\mu,\mathbb{G}$ and the function
$a_{m,r}^{\sharp}\ $such that%
\[
\left\Vert X_{i}X_{j}u\right\Vert _{L^{p}\left(  \Omega_{m}\right)
}+\left\Vert X_{i}u\right\Vert _{L^{p}\left(  \Omega_{m}\right)  }\leq
c\left\{  \left\Vert Lu\right\Vert _{L^{p}\left(  \Omega\right)  }+\left\Vert
u\right\Vert _{L^{p}\left(  \Omega\right)  }\right\}
\]
for $i,j=1,2,...,q$ and any $u\in W_{X}^{2,p}\left(  \Omega\right)  $.
\end{theorem}

What we will actually prove here is the basic step towards the above theorem, namely:

\begin{theorem}
\label{main result proved}Under the previous assumptions, for any $\Omega
_{m}\Subset\Omega$ and $p\in\left(  1,\infty\right)  $ the set $\Omega_{m}$
can be covered with a finite number of balls $B_{R}\left(  x_{i}\right)  $
such that for every $u\in C_{0}^{\infty}\left(  B_{R}\right)  $
\[
\sum_{i,j=1}^{q}\left\Vert X_{i}X_{j}u\right\Vert _{L^{p}\left(  B_{R}\right)
}\leq c\left\Vert Lu\right\Vert _{L^{p}\left(  B_{R}\right)  }%
\]
where the constant $c$ depends on $\Omega,\Omega_{m},p,\mu,\mathbb{G}$ and the
function $a_{m,r}^{\sharp}.$
\end{theorem}

The proof of Theorem \ref{main result proved} is where the different real
analysis approach of this paper with respect to \cite{bb1} plays its role.
Proving Theorem \ref{Final result} starting with Theorem
\ref{main result proved} is mainly a matter of cutoff functions and
interpolation inequalities for Sobolev norms, which can be performed exactly
like in \cite{bb1} and therefore will not be repeated here.

\section{Local estimates for the model operator}

We start with several a priori estimates for the operator $\overline{L}$,
defined as in (\ref{constant coefficient operator}) with \emph{constant}
$\left\{  \overline{a}_{ij}\right\}  $. The constants in our estimates will
depend on this matrix only through the number $\mu$. Recall that the operator
$\overline{L}$, which in our context is the analog of the constant coefficient
operator in the elliptic case, is hypoelliptic, $2$-homogeneous and
translation invariant on $\mathbb{G}$.

\begin{lemma}
\label{Lemma 1}For any $u\in C^{\infty}(\mathbb{R}^{n})$ and $R>0$, let $h\in
W_{X}^{1,2}(B_{R})$ be the weak solution to
\begin{equation}
\left\{
\begin{array}
[c]{ll}%
\overline{L}h=0 & \text{in }B_{R}\\
h=u & \text{on }\partial B_{R}.
\end{array}
\right.  \label{DP1}%
\end{equation}
(Here $B_{R}$ stands for $B_{R}\left(  0\right)  $). Then $h\in C^{\infty
}(B_{R})$ and if $R\geq4\Lambda^{2}$, where from now on $\Lambda$ is the
constant appearing in Poincar\'{e}'s inequality (Thm. \ref{poincare}), the
following holds:
\begin{equation}
\sup_{B_{1}}\left\vert X_{i}X_{j}X_{k}h\right\vert \leq c\sum_{i,j=1}^{q}\Vert
X_{i}X_{j}u\Vert_{L^{1}(B_{R})} \label{third derivatives}%
\end{equation}
for all $i,j,k=1,...,q$. The constant $c$ only depends on $\mathbb{G},\mu$, in
particular it is independent of $u$.
\end{lemma}

\begin{proof}
Let $w\in W_{X}^{1,2}(B_{R})$ be the unique weak solution to the Dirichlet
problem%
\begin{equation}
\left\{
\begin{array}
[c]{ll}%
\overline{L}w=-\overline{L}u & \text{in }B_{R}\\
w=0 & \text{on }\partial B_{R}%
\end{array}
\right. \nonumber
\end{equation}
and let $h=u+w.$ Then $h$ solves (\ref{DP1}) and, since $\overline{L}$ is
hypoelliptic in $\mathbb{R}^{n}$ and $-\overline{L}u\in C^{\infty}(B_{R})$,
$h\in C^{\infty}(B_{R})$.

To prove (\ref{third derivatives}), let us now assume $R\geq4\Lambda^{2}$ (in
particular, $R>4)$ and let us apply the subelliptic estimates (Thm.
\ref{subelliptic estimate}) with cutoff functions $\eta,\eta_{1}\in
C_{0}^{\infty}(B_{2})$, $\eta_{1}=1$ in sprt $\eta$:%
\[
\Vert\eta h\Vert_{H^{\sigma+\epsilon}}\leq c\left\{  \Vert\eta_{1}\overline
{L}h\Vert_{H^{\sigma}}+\Vert\eta_{1}h\Vert_{H^{-\tau}}\right\}  .
\]
Then since $\overline{L}h=0$ in $B_{R}$, taking $\tau=0$ and $\sigma$ large
enough we have%
\[
\sup_{B_{1}}\left\vert X_{i}X_{j}X_{k}h\right\vert \leq c\Vert\eta
h\Vert_{H^{\sigma+\epsilon}}\leq c\Vert h\Vert_{L^{2}(B_{2})},
\]
where the first inequality follows by the classical Sobolev embedding theorems.

Then, it is enough to prove that
\begin{equation}
\Vert h\Vert_{L^{2}(B_{2})}\leq c\sum_{i,j=1}^{q}\Vert X_{i}X_{j}u\Vert
_{L^{1}(B_{R})}. \label{Lemma1_local1}%
\end{equation}

Let $\varphi\in C^{\infty}(\mathbb{R}^{n})$ such that $\varphi(x)=1$ if $\Vert
x\Vert\geq3.5$ and $\varphi(x)=0$ if $\Vert x\Vert\leq3$ and define
\[
v=h-\varphi u.
\]
Then $v\in C^{\infty}(B_{R})$ and%
\[
\overline{L}v=\overline{L}(-\varphi u)=-\varphi\overline{L}u-u\overline
{L}\varphi-2\sum_{i,j=1}^{q}\overline{a}_{ij}X_{i}\varphi\,X_{j}u=:-g.
\]
Also, since $h-u\in W_{X,0}^{1,2}(B_{R})$ and $\varphi=1$ near $\partial
B_{R}$, we have $v\in W_{X,0}^{1,2}(B_{R})$.

On the other hand, for
\[
f=\left(  \left\vert \varphi\overline{L}u\right\vert +\left\vert u\overline
{L}\varphi\right\vert +2\left\vert \sum_{i,j=1}^{q}\overline{a}_{ij}%
X_{i}\varphi\,X_{j}u\right\vert \right)  \chi_{B_{R}}%
\]
defined in $\mathbb{R}^{n}$, and $\Gamma_{\overline{a}}$ the global
homogeneous fundamental solution of $\overline{L}$, let
\[
w(x)=-\int_{\mathbb{R}^{n}}\Gamma_{\overline{a}}(y^{-1}\circ x)f(y)dy.
\]
Then $-\overline{L}w=f$ in strong sense (that is, $w\in W_{X}^{2,2}\left(
B_{R}\right)  $ and $-\overline{L}w\left(  x\right)  =f\left(  x\right)  $ for
a.e. $x\in B_{R}$) and then also in the weak sense, and $w\geq0$ in
$\mathbb{R}^{n}$ (since both $-\Gamma_{\overline{a}}$ and $f$ are
nonnegative). Hence the functions $v,w$ satisfy, in weak sense,%
\begin{align*}
&  \left\{
\begin{array}
[c]{ll}%
\overline{L}(v-w)=f-g\geq0 & \text{in }B_{R}\\
v-w\leq0 & \text{on }\partial B_{R}%
\end{array}
\right. \\
&  \left\{
\begin{array}
[c]{ll}%
\overline{L}(-v-w)=g+f\geq0 & \text{in }B_{R}\\
-v-w\leq0 & \text{on }\partial B_{R}%
\end{array}
\right.
\end{align*}
and since $\left\vert g\right\vert \leq f$, by the maximum principle
(Proposition~\ref{maximum principle}) we conclude $\left\vert v\right\vert
\leq w$ in $B_{R}$.

Now for $x\in B_{2}$, since $\varphi(x)=0$ if $\left\Vert x\right\Vert \leq3$
and $f\left(  x\right)  \neq0$ only for $3\leq\left\Vert x\right\Vert \leq
R$,
\begin{align*}
\left\vert h(x)\right\vert  &  =\left\vert \left(  v+\varphi u\right)  \left(
x\right)  \right\vert =\left\vert v(x)\right\vert \leq w(x)=-\int
_{\mathbb{R}^{n}}\Gamma_{\overline{a}}(y^{-1}\circ x)f(y)dy\ \\
&  =-\int_{B_{R}\setminus B_{3}}\Gamma_{\overline{a}}(y^{-1}\circ x)f(y)dy.
\end{align*}
On the other hand, for $x\in B_{2}$ and $y\in B_{R}\setminus B_{3}$ the
function $\Gamma_{\overline{a}}(y^{-1}\circ x)$ is bounded. Actually, by
Proposition \ref{bounded not depending on the coeficients} and
\eqref{inequality of the distance}
\[
0\leq-\Gamma_{\overline{a}}(y^{-1}\circ x)\leq\frac{c}{\left\Vert y^{-1}\circ
x\right\Vert ^{Q-2}}\leq\frac{c}{\left(  \left\Vert y\right\Vert -\left\Vert
x\right\Vert \right)  ^{Q-2}}\leq c.
\]
Hence%
\[
\left\vert h(x)\right\vert \leq c\Vert f\Vert_{1}\ \leq c\left\{
\Vert\overline{L}u\Vert_{L^{1}(B_{R})}+\Vert u\Vert_{L^{1}(B_{3.5})}%
+\sum_{j=1}^{q}\Vert X_{j}u\Vert_{L^{1}(B_{3.5})}\right\}
\]
which in particular gives%
\begin{equation}
\Vert h\Vert_{L^{1}(B_{2})}\leq c\left\{  \sum_{i,j=1}^{q}\Vert X_{i}%
X_{j}u\Vert_{L^{1}(B_{R})}+\Vert u\Vert_{L^{1}(B_{4})}+\sum_{j=1}^{q}\Vert
X_{j}u\Vert_{L^{1}(B_{4})}\right\}  . \label{Lemma1_local2}%
\end{equation}
In order to prove (\ref{Lemma1_local1}) we should remove from the right-hand
side of (\ref{Lemma1_local2}) the terms in $u$ and $X_{j}u$. To this aim, let%
\[
\widetilde{u}\left(  x\right)  =u\left(  x\right)  +c_{0}+\sum_{i=1}^{q}%
c_{i}x_{i}%
\]
for some constants $c_{i}$, $i=0,1,2,...,q$ that we can choose so that%
\begin{align*}
\int_{B_{4}}\widetilde{u}\left(  x\right)  dx  &  =0\\
\int_{B_{4\Lambda}}X_{i}\widetilde{u}\left(  x\right)  dx  &  =0\text{ for
}i=1,2,...,q.
\end{align*}
Namely, since for $i=1,2,...,q$ the vector fields $X_{i}$ have the structure%
\[
X_{i}=\partial_{x_{i}}+\sum_{j=q+1}^{n}b_{ij}\left(  x\right)  \partial
_{x_{j}},
\]
so that $X_{j}\widetilde{u}=X_{j}u+c_{j}$, we can choose%
\begin{align*}
c_{i}  &  =-\frac{1}{\left\vert B_{4\Lambda}\right\vert }\int_{B_{4\Lambda}%
}X_{i}u\left(  x\right)  dx\text{ for }i=1,2,...,q\\
c_{0}  &  =-\frac{1}{\left\vert B_{4}\right\vert }\left(  \int_{B_{4}}u\left(
x\right)  dx+\sum_{i=1}^{q}c_{i}\int_{B_{4}}x_{i}dx\right)  .
\end{align*}
For this choice of $c_{i}$, $i=0,1,2,...,q$ and $\widetilde{u}$, we can now
repeat the above proof defining $\widetilde{h}$ as the solution to
\begin{equation}
\left\{
\begin{array}
[c]{ll}%
\overline{L}\widetilde{h}=0 & \text{in }B_{R}\\
\widetilde{h}=\widetilde{u} & \text{on }\partial B_{R}%
\end{array}
\right. \nonumber
\end{equation}
(with $R\geq4\Lambda^{2}$ as before). Clearly, one simply has%
\[
\widetilde{h}\left(  x\right)  =h\left(  x\right)  +c_{0}+\sum_{i=1}^{q}%
c_{i}x_{i}%
\]
and%
\begin{align*}
\sup_{B_{1}}\left\vert X_{i}X_{j}X_{k}\widetilde{h}\right\vert  &  \leq
c\Vert\widetilde{h}\Vert_{L^{2}(B_{2})}\\
&  \leq c\left\{  \sum_{i,j=1}^{q}\Vert X_{i}X_{j}\widetilde{u}\Vert
_{L^{1}(B_{R})}+\Vert\widetilde{u}\Vert_{L^{1}(B_{4})}+\sum_{j=1}^{q}\Vert
X_{j}\widetilde{u}\Vert_{L^{1}(B_{4})}\right\}  .
\end{align*}
Next, note that $X_{i}X_{j}\widetilde{u}=X_{i}\left(  X_{j}u+c_{j}\right)
=X_{i}X_{j}u$ and by Poincar\'{e}'s inequality (Thm.~\ref{poincare})
\begin{align*}
&  \Vert\widetilde{u}\Vert_{L^{1}(B_{4})}+\sum_{j=1}^{q}\Vert X_{j}%
\widetilde{u}\Vert_{L^{1}(B_{4})}\\
&  =\int_{B_{4}}\left\vert \widetilde{u}\left(  x\right)  -\widetilde
{u}_{B_{4}}\right\vert dx+\sum_{j=1}^{q}\Vert X_{j}\widetilde{u}\Vert
_{L^{1}(B_{4})}\\
&  \leq c\sum_{i=1}^{q}\int_{B_{4\Lambda}}\left\vert X_{i}\widetilde{u}\left(
x\right)  \right\vert dx+\sum_{j=1}^{q}\Vert X_{j}\widetilde{u}\Vert
_{L^{1}(B_{4\Lambda})}%
\end{align*}%
\begin{align*}
&  =c\sum_{i=1}^{q}\int_{B_{4\Lambda}}\left\vert X_{i}\widetilde{u}\left(
x\right)  -X_{i}\widetilde{u}_{B_{4\Lambda}}\right\vert dx\\
&  \leq c\sum_{i,j=1}^{q}\int_{B_{4\Lambda^{2}}}\left\vert X_{j}%
X_{i}\widetilde{u}\left(  x\right)  \right\vert dx\\
&  =c\sum_{i,j=1}^{q}\int_{B_{4\Lambda^{2}}}\left\vert X_{j}X_{i}u\left(
x\right)  \right\vert dx.
\end{align*}
Also, $X_{i}X_{j}X_{k}\widetilde{h}=X_{i}X_{j}X_{k}h$ hence%
\[
\sup_{B_{1}}\left\vert X_{i}X_{j}X_{k}h\right\vert \leq c\sum_{i,j=1}^{q}\Vert
X_{i}X_{j}u\Vert_{L^{1}\left(  B_{4\Lambda^{2}}\right)  }\leq c\sum
_{i,j=1}^{q}\Vert X_{i}X_{j}u\Vert_{L^{1}\left(  B_{R}\right)  }%
\]
and we are done.
\end{proof}

\begin{lemma}
\label{Lemma2}For any $k\geq4\Lambda^{3}$, $r>0$ , $u\in C^{\infty}%
(\mathbb{R}^{n})$ and $h$ the weak solution to%
\begin{equation}
\left\{
\begin{array}
[c]{ll}%
\overline{L}h=0 & \text{ in }B_{kr}\\
h=u & \text{on }\partial B_{kr}%
\end{array}
\right. \nonumber
\end{equation}
we have that for $i,j=1,2,...,q$%
\begin{equation}
\frac{1}{|B_{r}|}\int_{B_{r}}|X_{i}X_{j}h\left(  x\right)  -(X_{i}%
X_{j}h)_{B_{r}}|dx\leq\frac{c}{k}\sum_{i,j=1}^{q}\frac{1}{|B_{kr}|}%
\int_{B_{kr}}|X_{i}X_{j}u\left(  x\right)  |dx, \label{Lemma2_ineq}%
\end{equation}
where the constant $c$ depends on $\mathbb{G}$ and $\mu$, but is independent
from $k$ and $r$.
\end{lemma}

\begin{proof}
It is enough to prove the result for $r=1$. Namely, if we define
$\widetilde{h}\left(  x\right)  =h(D(r)(x))$ and $\widetilde{u}\left(
x\right)  =u(D(r)(x))$, using the 1-homogenety of $X_{i}$, by dilations we
have
\begin{align*}
\frac{1}{|B_{r}|}\int_{B_{r}}|X_{i}X_{j}h(x)|dx  &  =\frac{1}{r^{Q}|B_{1}%
|}\int_{B_{1}}|\left(  X_{i}X_{j}h\right)  (D_{r}(y))|r^{Q}dy\ \\
&  =\frac{1}{|B_{1}|}r^{-2}\int_{B_{1}}|X_{i}X_{j}\widetilde{h}(y)|dy
\end{align*}
Analogously, we obtain
\[
\frac{1}{|B_{r}|}\int_{B_{r}}|X_{i}X_{j}h(x)-(X_{i}X_{j}h)_{B_{r}}|dx=\frac
{1}{|B_{1}|}r^{-2}\int_{B_{1}}|X_{i}X_{j}\widetilde{h}(y)-(X_{i}%
X_{j}\widetilde{h})_{B_{1}}|dy
\]
and%

\[
\frac{1}{|B_{kr}|}\int_{B_{kr}}|X_{i}X_{j}u(x)|dx=\frac{1}{|B_{k}|}r^{-2}%
\int_{B_{k}}|X_{i}X_{j}\widetilde{u}(y)|dy
\]
hence if the result holds for $r=1$ it holds for every $r>0$.

Now, for $k\geq4\Lambda^{3}$, let $h\in W_{X}^{1,2}(B_{k})$ satisfy
\begin{equation}
\left\{
\begin{array}
[c]{ll}%
\overline{L}h=0 & \text{in }B_{k}\\
h=u & \text{on }\partial B_{k}.
\end{array}
\right.  \label{Lemma1_local4}%
\end{equation}
Let us assume that for every $s,i,j=1,...,q$ and $x\in B_{\Lambda}$ ,
\begin{equation}
\left\vert X_{s}X_{i}X_{j}h\left(  x\right)  \right\vert \leq\frac{c}{k}%
\sum_{i,j}^{q}\frac{1}{|B_{k}|}\int_{B_{k}}|X_{i}X_{j}u(x)|dx
\label{Lemma1_local3}%
\end{equation}
(with $c$ independent of $k$) and let us prove (\ref{Lemma2_ineq}) for $r=1$.

By Theorem \ref{poincare},%
\begin{align*}
\frac{1}{|B_{1}|}\int_{B_{1}}|X_{i}X_{j}h(x)-(X_{i}X_{j}h)_{B_{1}}|dx  &
\leq\frac{c}{|B_{\Lambda}|}\sum_{s=1}^{q}\int_{B_{\Lambda}}|X_{s}X_{i}%
X_{j}h(x)|dx\ \\
&  \leq c\sum_{s=1}^{q}\sup_{B_{\Lambda}}\left\vert X_{s}X_{i}X_{j}%
h\right\vert \ \\
&  \leq\frac{c}{k}\sum_{i,j}^{q}\frac{1}{|B_{k}|}\int_{B_{k}}|X_{i}%
X_{j}u(x)|dx,
\end{align*}
which is the assertion for $r=1$.

It remains to prove (\ref{Lemma1_local3}). To do that, for $x\in
B_{4\Lambda^{2}}$ we define $\widetilde{h}(x)=h(D(k/4\Lambda^{2})(x))$ and
$\widetilde{u}(x)=u(D(k/4\Lambda^{2})(x))$. Then $\overline{L}\widetilde{h}=0$
in $B_{4\Lambda^{2}}$ with boundary condition $\widetilde{u}$ and we can apply
Lemma \ref{Lemma 1}, which jointly with dilations and homogenety gives for
$x\in B_{1}$%
\begin{align*}
&  \left(  \frac{k}{4\Lambda^{2}}\right)  ^{3}\left\vert \left(  X_{s}%
X_{i}X_{j}h\right)  (D(k/4\Lambda^{2})(x))\right\vert =\left\vert X_{s}%
X_{i}X_{j}\widetilde{h}(x)\right\vert \ \ \\
&  \leq c\sum_{i,j=1}^{q}\int_{B_{4\Lambda^{2}}}\left\vert X_{i}%
X_{j}\widetilde{u}(x)\right\vert dx\
\end{align*}%
\begin{align*}
&  =c\left\vert B_{4\Lambda^{2}}\right\vert \sum_{i,j=1}^{q}\frac
{1}{\left\vert B_{4\Lambda^{2}}\right\vert }\int_{B_{4\Lambda^{2}}}|X_{i}%
X_{j}(u(D(k/4\Lambda^{2})(x))|dx\ \\
&  =c\left(  \frac{k}{4\Lambda^{2}}\right)  ^{2}\sum_{i,j=1}^{q}\frac
{1}{\left\vert B_{4\Lambda^{2}}\right\vert }\int_{B_{4\Lambda^{2}}}|\left(
X_{i}X_{j}u\right)  (D(k/4\Lambda^{2})(x))|dx\ \\
&  =c\left(  \frac{k}{4\Lambda^{2}}\right)  ^{2}\sum_{i,j=1}^{q}\frac
{1}{|B_{k}|}\int_{B_{k}}|X_{i}X_{j}u(y)|dy.
\end{align*}
Hence, for $x\in B_{1},$
\[
\left\vert \left(  X_{s}X_{i}X_{j}h\right)  (D(k/4\Lambda^{2})(x))\right\vert
\leq\frac{c}{k}\sum_{i,j=1}^{q}\frac{1}{|B_{k}|}\int_{B_{k}}|X_{i}%
X_{j}u(x)|dx.
\]
But, since $x$ ranges in $B_{1}$, the point $y=D(k/4\Lambda^{2})(x)$ ranges in
$B_{k/4\Lambda^{2}}\supset B_{\Lambda}$ (because $k/4\Lambda^{2}\geq\Lambda$)
and the Lemma is proved.
\end{proof}

The next Lemma can be of independent interest:

\begin{lemma}
\label{Lemma ex BB1}Let $p\in\left(  1,\infty\right)  $. There exists a
constant $c$ depending on $p,\mathbb{G},\mu$ such that for any $r>0,k\geq
2,v\in W_{X,0}^{1,2}\left(  B_{kr}\right)  $ the following holds:%
\[
\left\Vert D^{2}v\right\Vert _{L^{p}\left(  B_{r}\right)  }\leq ck^{2}%
\left\Vert \overline{L}v\right\Vert _{L^{p}\left(  B_{kr}\right)  }.
\]

\end{lemma}

Before proving this result, let us explain why it is not trivial. From the
local estimates proved by Folland \cite{fo} it is known that for any $v\in
W_{X}^{1,2}\left(  B_{kr}\right)  $%
\[
\left\Vert D^{2}v\right\Vert _{L^{p}\left(  B_{r}\right)  }\leq c\left(
\left\Vert \overline{L}v\right\Vert _{L^{p}\left(  B_{kr}\right)  }+\left\Vert
Dv\right\Vert _{L^{p}\left(  B_{kr}\right)  }+\left\Vert v\right\Vert
_{L^{p}\left(  B_{kr}\right)  }\right)  .
\]
Also, for $v\in C_{0}^{\infty}\left(  B_{r}\right)  $ one can prove%
\[
\left\Vert D^{2}v\right\Vert _{L^{p}\left(  B_{r}\right)  }\leq c\left\Vert
\overline{L}v\right\Vert _{L^{p}\left(  B_{r}\right)  }.
\]
The nontrivial fact, in the subelliptic context (where $L^{p}$ estimates up to
the boundary are unknown), is removing the $L^{p}$ norm of $v$ from the right
hand side under the weak vanishing condition $v\in W_{0}^{1,2}\left(
B_{kr}\right)  .$

\bigskip

\begin{proof}
For any $\sigma\in(\frac{1}{2},1)$, we can construct (see \cite{bb1} for
details) a cutoff function $\varphi_{\sigma}\in C_{0}^{\infty}(\mathbb{R}%
^{n})$ satisfying: $\varphi_{\sigma}=1$ on $B_{\sigma r}$,
$\operatorname{sprt}\varphi_{s}\subset B_{\sigma^{\prime}r}$, where
$\sigma^{\prime}=\frac{(1+\sigma)}{2}$,
\begin{align*}
\left\vert X_{j}\varphi_{\sigma}\right\vert  &  \leq\frac{c}{(1-\sigma)r}\\
\left\vert X_{i}X_{j}\varphi_{\sigma}\right\vert  &  \leq\frac{c}%
{(1-\sigma)^{2}r^{2}}.
\end{align*}
Let us define two cutoff functions $\varphi_{1},\varphi_{2}$ corresponding to
$\sigma_{1}\in(\frac{1}{2},1),\sigma_{2}=\sigma_{1}^{\prime},$ and let
$\sigma_{3}=\sigma_{2}^{\prime}.$ We can apply Folland's local estimates for
the model operator (see \cite[Theorem 4.9]{fo}) to $v\varphi_{1},$ so that%
\[
\Vert X_{i}X_{j}(v\varphi_{1})\Vert_{L^{p}(B_{\sigma_{2}r})}\leq
c\Vert\overline{L}(v\varphi_{1})\Vert_{L^{p}(B_{\sigma_{2}r})}.
\]
Then, expanding the operator $\overline{L}(v\varphi_{1})$, using the estimate
for the derivatives of $\varphi_{1}$ and multiplying by $(1-\sigma_{1}%
)^{2}r^{2}$ in both sides, we have
\begin{align}
(1-\sigma_{1})^{2}r^{2}\Vert X_{i}X_{j}v\Vert_{L^{p}(B_{\sigma_{1}r})}  &
\leq cr^{2}\Vert\overline{L}v\Vert_{L^{p}(B_{r})}\ \nonumber\\
&  +c(1-\sigma_{1})r\Vert X_{i}v\Vert_{L^{p}(B_{\sigma_{2}r})}\ \nonumber\\
&  +c\Vert v\Vert_{L^{p}(B_{r})}. \label{second_derivatives_of_v}%
\end{align}

In order to estimate $(1-\sigma_{1})r\Vert X_{i}v\Vert_{L^{p}(B_{\sigma_{2}%
r})}$, let us apply Proposition~\ref{interpolation inequality} to
$v\varphi_{2}$. We have%
\begin{align*}
\Vert X_{i}v\Vert_{L^{p}(B_{\sigma_{2}r})}  &  \leq\varepsilon\left\{  \Vert
X_{i}^{2}v\Vert_{L^{p}(B_{\sigma_{3}r})}+\frac{1}{(1-\sigma_{2})r}\Vert
X_{i}v\Vert_{L^{p}(B_{\sigma_{3}r})}\right.  \ \\
&  \left.  +\frac{1}{(1-\sigma_{2})^{2}r^{2}}\Vert v\Vert_{L^{p}(B_{\sigma
_{3}r})}\right\}  +\frac{2}{\varepsilon}\Vert v\Vert_{L^{p}(B_{\sigma_{3}r})}.
\end{align*}
Now, taking $\varepsilon=(1-\sigma_{2})r\delta$ for some $\delta$ and using
the fact that $\frac{1-\sigma}{1-\sigma^{\prime}}=\frac{1}{2}$ we obtain
\begin{align*}
(1-\sigma_{2})r\Vert X_{i}v\Vert_{L^{p}(B_{\sigma_{2}r})}  &  \leq
c\delta(1-\sigma_{3})^{2}r^{2}\Vert X_{i}^{2}v\Vert_{L^{p}(B_{\sigma_{3}r}%
)}\ \\
&  +c\delta(1-\sigma_{3})r\Vert X_{i}v\Vert_{L^{p}(B_{\sigma_{3}r})}\ \\
&  +c\delta\Vert v\Vert_{L^{p}(B_{\sigma_{3}r})}+\frac{2}{\delta}\Vert
v\Vert_{L^{p}(B_{\sigma_{3}r})},
\end{align*}
which, letting%
\[
\phi_{k}=\sup_{\sigma\in(\frac{1}{2},1)}(1-\sigma)^{k}r^{k}\Vert D^{k}%
v\Vert_{L^{p}(B_{\sigma r})}%
\]
implies that
\[
\phi_{1}\leq c\delta(\phi_{2}+\phi_{1}+\Vert v\Vert_{L^{p}(B_{r})})+\frac
{c}{\delta}\Vert v\Vert_{L^{p}(B_{r})},
\]
and taking $\delta$ small enough we have
\[
\phi_{1}\leq c\delta\phi_{2}+C\Vert v\Vert_{L^{p}(B_{r})}.
\]
Finally, inserting this in (\ref{second_derivatives_of_v}) and taking the
supremum on $\sigma_{1}$ we have%
\[
\phi_{2}\leq cr^{2}\Vert\overline{L}v\Vert_{L^{p}(B_{r})}+c\Vert v\Vert
_{L^{p}(B_{r})},
\]
which can be read as
\begin{equation}
r^{2}\Vert X_{i}X_{j}v\Vert_{L^{p}(B_{r})}\leq cr^{2}\Vert\overline{L}%
v\Vert_{L^{p}(B_{kr})}+c\Vert v\Vert_{L^{p}(B_{kr})} \label{BB1 local}%
\end{equation}
for $r>0$, $k>2$ and for some $c$ depending on $p,\mathbb{G},\mu$.

On the other hand, the function
\[
w\left(  x\right)  =-\int_{B_{kr}}\Gamma_{\overline{a}}\left(  x^{-1}\circ
y\right)  \left\vert f\left(  y\right)  \right\vert dy
\]
solves%
\[
\left\{
\begin{array}
[c]{l}%
\overline{L}w=-\left\vert f\right\vert \text{ in }B_{kr}\\
w\geq0\text{ on }\partial B_{kr}%
\end{array}
\right.
\]
and taking $f=\overline{L}v\cdot\chi_{B_{kr}}$, by the same reasoning of the
proof of Lemma~\ref{Lemma 1}, the maximum principle implies $\left\vert
v\right\vert \leq w\text{ in }B_{kr}\text{.}$ Then, by
Proposition~\ref{bounded not depending on the coeficients}%
\begin{align*}
\left\vert v(x)\right\vert  &  \leq w\left(  x\right)  \leq c\int_{B_{kr}%
}\frac{1}{\Vert x^{-1}\circ y\Vert^{Q-2}}|f(y)|dy\ \\
&  \leq c\sum_{s=0}^{\infty}\int_{\frac{2kr}{2^{s+1}}\leq\Vert x^{-1}\circ
y\Vert\leq\frac{2kr}{2^{s}}}\frac{1}{\Vert x^{-1}\circ y\Vert^{Q-2}%
}|f(y)|dy\ \\
&  \leq c\sum_{s=0}^{\infty}\left(  \frac{2^{s+1}}{2kr}\right)  ^{Q-2}%
\int_{\Vert x^{-1}\circ y\Vert\leq\frac{2kr}{2^{s}}}|f(y)|dy\ \\
&  \leq c(kr)^{2}\sum_{s=0}^{\infty}\frac{1}{2^{2s}}Mf(x),
\end{align*}
and by (\ref{boundedness of the maximal})
\[
\Vert v\Vert_{L^{p}(B_{kr})}\leq(kr)^{2}\Vert\overline{L}v\Vert_{L^{p}%
(B_{kr})}%
\]
which inserted in (\ref{BB1 local}) gives us the result.
\end{proof}

\begin{lemma}
\label{Lemma3}Let $p\in(1,\infty)$. Then there exists a constant $c$ depending
on $p,\mathbb{G},\mu$ such that for $k\geq4\Lambda^{3}$, $r>0$ and $u\in
C^{\infty}(\mathbb{R}^{n})$
\begin{align*}
&  \frac{1}{|B_{r}|}\int_{B_{r}}|X_{i}X_{j}u\left(  x\right)  -(X_{i}%
X_{j}u)_{B_{r}}|dx\\
&  \leq\frac{c}{k}\sum_{i,j=1}^{q}\frac{1}{|B_{kr}|}\int_{B_{kr}}|X_{i}%
X_{j}u\left(  x\right)  |dx+ck^{2+Q/p}\left(  \frac{1}{|B_{kr}|}\int_{B_{kr}%
}|\overline{L}u\left(  x\right)  |^{p}dx\right)  ^{1/p}.
\end{align*}

\end{lemma}

\begin{proof}
For $u$ and $k$ as in the statement, let $h$ be the solution to
\[
\left\{
\begin{array}
[c]{ll}%
\overline{L}h=0 & \text{in }B_{kr}\\
h=u & \text{on }\partial B_{kr},
\end{array}
\right.  \label{DPh}%
\]
then%
\begin{align*}
\frac{1}{|B_{r}|}\int_{B_{r}}|X_{i}X_{j}u(x)-(X_{i}X_{j}u)_{B_{r}}|dx  &
\leq\frac{1}{|B_{r}|}\int_{B_{r}}|X_{i}X_{j}u(x)-X_{i}X_{j}h(x)|dx\ \\
&  +\frac{1}{|B_{r}|}\int_{B_{r}}|X_{i}X_{j}h(x)-(X_{i}X_{j}h)_{B_{r}}|dx\ \\
&  +\frac{1}{|B_{r}|}\int_{B_{r}}|(X_{i}X_{j}h)_{B_{r}}-(X_{i}X_{j}u)_{B_{r}%
}|dx\ \\
&  \equiv A+B+C.
\end{align*}

By Lemma \ref{Lemma2} we have
\[
B\leq\frac{c}{k}\sum_{i,j=1}^{q}\frac{1}{|B_{kr}|}\int_{B_{kr}}|X_{i}%
X_{j}u\left(  x\right)  |dx.
\]

As to $C$, since $(X_{i}X_{j}h)_{B_{r}}-(X_{i}X_{j}u)_{B_{r}}=(X_{i}%
X_{j}h-X_{i}X_{j}u)_{B_{r}}$ it is enough to estimate the term $A$.

Applying Lemma~\ref{Lemma ex BB1} to the weak solution $v$ of the problem
\[
\left\{
\begin{array}
[c]{ll}%
\overline{L}v=\overline{L}u & \text{in }B_{kr}\\
v=0 & \text{on }\partial B_{kr}%
\end{array}
\right.  \label{DPu-h}%
\]
we have%
\[
\Vert X_{i}X_{j}v\Vert_{L^{p}(B_{r})}\leq ck^{2}\Vert\overline{L}v\Vert
_{L^{p}(B_{kr})}.\label{second derivatives estimates}%
\]
Then, by H\"{o}lder inequality we obtain
\begin{align*}
\frac{1}{|B_{r}|}\int_{B_{r}}|X_{i}X_{j}v(x)|dx  &  \leq\left(  \frac
{1}{|B_{r}|}\int_{B_{r}}|X_{i}X_{j}v(x)|^{p}dx\right)  ^{1/p}\ \\
&  \leq ck^{2}|B_{r}|^{-1/p}\left(  \int_{B_{kr}}|\overline{L}v(x)|^{p}%
dx\right)  ^{1/p}\ \\
&  =ck^{2}|B_{r}|^{-1/p}\left(  \int_{B_{kr}}|\overline{L}u(x)|^{p}dx\right)
^{1/p}\ \\
&  =ck^{2+Q/p}\left(  \frac{1}{|B_{kr}|}\int_{B_{kr}}|\overline{L}u\left(
x\right)  |^{p}dx\right)  ^{1/p}%
\end{align*}
and we are done.
\end{proof}

\section{Local estimates for operators with variable coefficients}

Let us now come to study the operator $L$ with variable $VMO_{loc}\left(
\Omega\right)  $ coefficients. The next theorem contains the key local
estimate involving $L$.

For a fixed domain $\Omega_{m}\Subset\Omega_{m+1}$, let us cover $\Omega_{m}$
with a finite number of balls $B_{R}$ with $R$ small enough ($R$ to be chosen
later). In the following theorem $B_{R}$ is one of these balls. The maximal
operator and the local sharp maximal operator which appear in the statement
are defined in \eqref{maximal} and \eqref{sharp maximal} respectively. The
function $a_{m,R}^{\sharp}$ ($VMO$ modulus of the coefficients $a_{ij}$) is
defined in (\ref{a sharp}).

\begin{theorem}
\label{Theorem 3.6}Let $p,\alpha,\beta\in(1,\infty)$ with $\alpha^{-1}%
+\beta^{-1}=1$ and $R\in(0,\infty)$. Then there exists a constant $c$
depending on $p,\alpha,\mathbb{G},\mu$ such that for any $u\in C_{0}^{\infty
}(B_{R})$ and $k\geq4\Lambda^{3}$ we have
\begin{align*}
(X_{i}X_{j}u)_{\Omega_{m+2},\Omega_{m+3}}^{\#}\left(  x\right)   &  \leq
\frac{c}{k}\sum_{i,j=1}^{q}M(X_{i}X_{j}u)\left(  x\right)  +ck^{2+Q/p}\left(
M(|Lu|^{p})\left(  x\right)  \right)  ^{1/p}\ \\
&  +ck^{2+Q/p}\left(  a_{m+2,R}^{\sharp}\right)  ^{1/\beta p}\left(
M(|X_{i}X_{j}u|^{p\alpha})\left(  x\right)  \right)  ^{1/\alpha p}\
\end{align*}
for every $x\in B_{R}$, $R<\varepsilon_{m+2}.$
\end{theorem}

The choice of bounding the local sharp maximal function relative to the
domains $\Omega_{m+2},\Omega_{m+3}$ is just for consistence with Theorem
\ref{Thm Fefferman Stein}. As will be apparent from the proof, we could bound
$(X_{i}X_{j}u)_{\Omega_{k},\Omega_{k+1}}^{\#}$ for any desired value of the
integer $k$.

\bigskip

\begin{proof}
Fix $k\geq4\Lambda^{3}$, $r\in\left(  0,\varepsilon_{m+2}\right)  $ and
$\overline{x}\in B_{R}$. Let $B_{r}$ be a ball containing $\overline{x}$. Let
$\overline{L}$ be a constant coefficients operator corresponding to a constant
matrix $\left\{  \overline{a}_{ij}\right\}  $ which will be chosen later,
depending on the values of $r$ and $k$, in the class of matrices satisfying
(\ref{ellipticity}). By Lemma \ref{Lemma3} we have that
\begin{align}
&  \frac{1}{|B_{r}|}\int_{B_{r}}|X_{i}X_{j}u\left(  x\right)  -(X_{i}%
X_{j}u)_{B_{r}}|dx\nonumber\\
&  \leq\frac{c}{k}\sum_{i,j=1}^{q}\frac{1}{|B_{kr}|}\int_{B_{kr}}|X_{i}%
X_{j}u\left(  x\right)  |dx+ck^{2+Q/p}\left(  \frac{1}{|B_{kr}|}\int_{B_{kr}%
}|\overline{L}u(x)|^{p}dx\right)  ^{1/p}\ \nonumber\\
&  \equiv A+B. \label{Thm3.6_1}%
\end{align}

To handle the term $B$, let us write%
\begin{equation}
\left(  \int_{B_{kr}}|\overline{L}u(x)|^{p}dx\right)  ^{1/p}\leq\left(
\int_{B_{kr}}|\overline{L}u(x)-Lu(x)|^{p}dx\right)  ^{1/p}+\left(
\int_{B_{kr}}|Lu(x)|^{p}dx\right)  ^{1/p} \label{Thm3.6_b}%
\end{equation}
with%
\begin{align}
&  \int_{B_{kr}}|\overline{L}u(x)-Lu(x)|^{p}dx\leq c\sum_{i,j=1}^{q}%
\int_{B_{kr}\cap B_{R}}|\overline{a}_{ij}-a_{ij}(x)|^{p}|X_{i}X_{j}%
u(x)|^{p}dx\ \nonumber\\
&  \leq c\sum_{i,j=1}^{q}\left(  \int_{B_{kr}\cap B_{R}}|\overline{a}%
_{ij}-a_{ij}(x)|^{p\beta}dx\right)  ^{1/\beta}\left(  \int_{B_{kr}\cap B_{R}%
}|X_{i}X_{j}u(x)|^{p\alpha}dx\right)  ^{1/\alpha}\ \nonumber\\
&  \equiv c\sum_{i,j=1}^{q}J_{2}^{1/\beta}\ J_{1}^{1/\alpha}. \label{Thm3.6_3}%
\end{align}

We have
\begin{equation}
J_{1}\leq\int_{B_{kr}}|X_{i}X_{j}u(x)|^{p\alpha}dx=c(kr)^{Q}\frac{1}{|B_{kr}%
|}\int_{B_{kr}}|X_{i}X_{j}u\left(  x\right)  |^{p\alpha}dx \label{Term J_1}%
\end{equation}
and since the coeficients $\overline{a}_{ij},a_{ij}$ are bounded by $1/\mu$ we
also have
\[
J_{2}\leq\mu^{-\beta p+1}\int_{B_{kr}\cap B_{R}}\left\vert a_{ij}\left(
x\right)  -\overline{a}_{ij}\right\vert dx.
\]
We now choose a particular constant matrix $\left\{  \overline{a}%
_{ij}\right\}  $, depending on the values of $r,k$, as follows
\[
\overline{a}_{ij}=\left\{
\begin{array}
[c]{ll}%
(a_{ij})_{B_{R}} & \text{if }kr\geq R\\
(a_{ij})_{B_{kr}} & \text{if }kr\leq R.
\end{array}
\right.
\]
Then, if $kr\geq R$%
\begin{equation}
J_{2}\leq c\int_{B_{R}}|a_{ij}(x)-(a_{ij})_{B_{R}}|dx\leq c|B_{R}%
|a_{R}^{\sharp}\leq cR^{Q}a_{R}^{\sharp}\leq c(kr)^{Q}a_{R}^{\sharp}
\label{kr biger than R}%
\end{equation}
while if $kr\leq R$%
\begin{equation}
J_{2}\leq c\int_{B_{kr}}|a_{ij}(x)-(a_{ij})_{B_{kr}}|dx\leq c|B_{kr}%
|a_{kr}^{\sharp}\leq c(kr)^{Q}a_{R}^{\sharp} \label{kr less than R}%
\end{equation}
where, here and in the rest of the proof, we write $a_{R}^{\sharp}$ for
$a_{m+2,R}^{\sharp}$.

In any case, by (\ref{Thm3.6_3}), (\ref{Term J_1}), (\ref{kr biger than R})
and (\ref{kr less than R}) we obtain
\begin{align*}
\int_{B_{kr}}|\overline{L}u(x)-Lu(x)|^{p}dx  &  \leq c\sum_{i,j=1}^{q}\left(
(kr)^{Q}a_{R}^{\sharp}\right)  ^{1/\beta}\left(  (kr)^{Q}(|X_{i}%
X_{j}u|^{p\alpha})_{B_{kr}}\right)  ^{1/\alpha}\ \\
&  =c(kr)^{Q}(a_{R}^{\sharp})^{1/\beta}\sum_{i,j=1}^{q}\left(  (|X_{i}%
X_{j}u|^{p\alpha})_{B_{kr}}\right)  ^{1/\alpha}%
\end{align*}
which inserted in (\ref{Thm3.6_b}) gives%
\begin{align*}
&  \left(  \frac{1}{|B_{kr}|}\int_{B_{kr}}|\overline{L}u(x)|^{p}dx\right)
^{1/p}\leq\left(  \frac{1}{|B_{kr}|}\int_{B_{kr}}|Lu(x)|^{p}dx\right)
^{1/p}\ \\
&  +c(a_{R}^{\sharp})^{1/\beta p}\sum_{i,j=1}^{q}\left(  \frac{1}{|B_{kr}%
|}\int_{B_{kr}}|X_{i}X_{j}u(x)|^{p\alpha}dx\right)  ^{1/\alpha p}.
\end{align*}
In turn, inserting this estimate in (\ref{Thm3.6_1}) we get
\begin{align*}
&  \frac{1}{|B_{r}|}\int_{B_{r}}|X_{i}X_{j}u\left(  x\right)  -(X_{i}%
X_{j}u)_{B_{r}}|dx\\
&  \leq\frac{c}{k}\sum_{i,j=1}^{q}\frac{1}{|B_{kr}|}\int_{B_{kr}}|X_{i}%
X_{j}u(x)|dx+ck^{2+Q/p}\left(  \frac{1}{|B_{kr}|}\int_{B_{kr}}|Lu(x)|^{p}%
dx\right)  ^{1/p}\ \\
&  +ck^{2+Q/p}\left(  a_{R}^{\sharp}\right)  ^{1/\beta p}\sum_{i,j=1}%
^{q}\left(  \frac{1}{|B_{kr}|}\int_{B_{kr}}|X_{i}X_{j}u(x)|^{p\alpha
}dx\right)  ^{1/\alpha p}\ \\
&  \leq\frac{c}{k}\sum_{i,j=1}^{q}M(X_{i}X_{j}u)\left(  \overline{x}\right)
+ck^{2+Q/p}(M(|Lu|^{p})(\overline{x}))^{1/p}\ \\
&  +ck^{2+Q/p}(a_{R}^{\#})^{1/\beta p}\sum_{i,j=1}^{q}\left(  M(|X_{i}%
X_{j}u|^{p\alpha})(\overline{x})\right)  ^{1/\alpha p}.
\end{align*}

Note that in this estimate the constant matrix does not appear any longer. The
constants $c$ are independent of $k,r$ and the estimate holds for any
$k\geq4\Lambda^{3}$ and $r>0.$ We can then take the supremum with respect to
$r\in\left(  0,\varepsilon_{m+2}\right)  $, getting%
\begin{align*}
&  (X_{i}X_{j}u)_{\Omega_{m+2},\Omega_{m+3}}^{\#}\left(  \overline{x}\right)
\leq\frac{N}{k}\sum_{i,j=1}^{q}M(X_{i}X_{j}u)\left(  \overline{x}\right)  \ \\
&  +Nk^{2+Q/p}\left\{  \left(  M(|Lu|^{p})\left(  \overline{x}\right)
\right)  ^{1/p}+\left(  a_{m+2,R}^{\sharp}\right)  ^{1/\beta p}\left(
M(|X_{i}X_{j}u|^{p\alpha})\left(  \overline{x}\right)  \right)  ^{1/\alpha
p}\right\}  \ .
\end{align*}

\end{proof}

We are now in position to give the:

\bigskip

\begin{proof}
[Proof of Theorem \ref{main result proved}]Assume that $B_{R}$ and $B_{\gamma
R}$ are as in the statement of Theorem \ref{Thm Fefferman Stein}. Fix
$p\in(1,\infty)$ and choose $\alpha,\beta,p_{1}\in(1,\infty)$ such that
$\alpha p_{1}<p$ and $\alpha^{-1}+\beta^{-1}=1.$ Apply Theorem
\ref{Theorem 3.6} to these $\alpha,\beta,p_{1}$ and the ball $B_{\gamma R}$
(but with $u\in C_{0}^{\infty}\left(  B_{R}\right)  $) writing, for $x\in
B_{\gamma R}$: \
\begin{align*}
(X_{i}X_{j}u)_{\Omega_{m+2},\Omega_{m+3}}^{\#}\left(  x\right)   &  \leq
\frac{c}{k}\sum_{i,j=1}^{q}M(X_{i}X_{j}u)\left(  x\right)  +ck^{2+Q/p_{1}%
}\left(  M(|Lu|^{p_{1}})\left(  x\right)  \right)  ^{1/p_{1}}\ \\
&  +ck^{2+Q/p_{1}}\left(  a_{m+2,\gamma R}^{\sharp}\right)  ^{1/\beta p_{1}%
}\sum_{i,j=1}^{q}\left(  M(|X_{i}X_{j}u|^{p_{1}\alpha})\left(  x\right)
\right)  ^{1/\alpha p_{1}}.
\end{align*}
Then, taking $L^{p}\left(  B_{\gamma R}\right)  $ norms of both sides we get%
\begin{align}
&  \left\Vert (X_{i}X_{j}u)_{\Omega_{m+2},\Omega_{m+3}}^{\#}\right\Vert
_{L^{p}\left(  B_{\gamma R}\right)  }\leq\frac{c}{k}\sum_{i,j=1}^{q}\left\Vert
M(X_{i}X_{j}u)\right\Vert _{L^{p}\left(  B_{\gamma R}\right)  }\nonumber\\
&  +ck^{2+Q/p_{1}}\left(  \int_{B_{\gamma R}}\left(  M(|Lu|^{p_{1}})\left(
x\right)  \right)  ^{p/p_{1}}dx\right)  ^{1/p}\nonumber\\
&  +ck^{2+Q/p_{1}}\left(  a_{m+2,\gamma R}^{\sharp}\right)  ^{1/\beta p_{1}%
}\sum_{i,j=1}^{q}\left(  \int_{B_{\gamma R}}\left(  M(|X_{i}X_{j}%
u|^{p_{1}\alpha})\left(  x\right)  \right)  ^{p/\alpha p_{1}}dx\right)
^{1/p}. \label{coroll estim 1}%
\end{align}
Note that, since $u\in C_{0}^{\infty}(B_{R}),$%
\[
\int_{B_{R}}X_{i}X_{j}u\left(  x\right)  dx=0.
\]
This follows from the structure of the vector fields $X_{i}$ in Carnot groups,
since
\[
X_{i}f=\sum_{j=1}^{n}b_{ij}\left(  x\right)  \partial_{x_{j}}f=\sum_{j=1}%
^{n}\partial_{x_{j}}\left(  b_{ij}\left(  x\right)  f\right)  .
\]
Hence we can apply Theorem \ref{Thm Fefferman Stein} writing%
\[
\sum_{i,j=1}^{q}\left\Vert X_{i}X_{j}u\right\Vert _{L^{p}\left(  B_{R}\right)
}\leq c\sum_{i,j=1}^{q}\left\Vert (X_{i}X_{j}u)_{\Omega_{m+1},\Omega_{m+2}%
}^{\#}\right\Vert _{L^{p}\left(  B_{\gamma R}\right)  }%
\]
applying the $p$, $p/p_{1}$ and $p/\alpha p_{1}$-maximal inequality
(\ref{boundedness of the maximal}) on the right hand side of
(\ref{coroll estim 1}) (recall that $u$ is compactly supported in $B_{R}$):%
\begin{align*}
&  \leq\frac{c}{k}\sum_{i,j=1}^{q}\left\Vert X_{i}X_{j}u\right\Vert
_{L^{p}\left(  B_{R}\right)  }\\
&  +ck^{2+Q/p_{1}}\left\{  \left\Vert Lu\right\Vert _{L^{p}\left(
B_{R}\right)  }+\left(  a_{m+2,\gamma R}^{\sharp}\right)  ^{1/\beta p_{1}}%
\sum_{i,j=1}^{q}\left\Vert X_{i}X_{j}u\right\Vert _{L^{p}\left(  B_{R}\right)
}\right\}  .
\end{align*}
Since this inequality holds for any $k\geq4\Lambda^{3},$ we can now choose $k$
so that $c/k<1/2,$ getting%
\[
\sum_{i,j=1}^{q}\left\Vert X_{i}X_{j}u\right\Vert _{L^{p}\left(  B_{R}\right)
}\leq c\left\Vert Lu\right\Vert _{L^{p}\left(  B_{R}\right)  }+c\left(
a_{m+2,\gamma R}^{\sharp}\right)  ^{1/\beta p_{1}}\sum_{i,j=1}^{q}\left\Vert
X_{i}X_{j}u\right\Vert _{L^{p}\left(  B_{R}\right)  }.
\]
Finally, exploiting the $VMO_{loc}$ assumption on the coefficients $a_{ij}$ we
can choose $R$ small enough to have $c\left(  a_{m+2,\gamma R}^{\sharp
}\right)  ^{1/\beta p_{1}}<1/2,$ so that%
\[
\sum_{i,j=1}^{q}\left\Vert X_{i}X_{j}u\right\Vert _{L^{p}\left(  B_{R}\right)
}\leq c\left\Vert Lu\right\Vert _{L^{p}\left(  B_{R}\right)  }%
\]
and we are done.
\end{proof}

\bigskip

\noindent\textsc{Dipartimento di Matematica}

\noindent\textsc{Politecnico di Milano}

\noindent\textsc{Via Bonardi 9, 20133 Milano, Italy}

\noindent\texttt{marco.bramanti@polimi.it}

\bigskip

\noindent\textsc{Instituto de Matematica Aplicada del Litoral (CONICET-UNL)}

\noindent\textsc{Departamento de Matematica (FIQ-UNL)}

\noindent\textsc{Santa Fe, Argentina}

\noindent\texttt{mtoschi@santafe-conicet.gov.ar}

\bigskip


\bigskip

\begin{thebibliography}{99}                                                                                               %


\bibitem {ADN}S. Agmon, A. Douglis, L. Nirenberg: \emph{Estimates near the
boundary of solutions of elliptic partial differential equations under general
boundary conditions. Part I,} Comm. Pure Applied Math., vol. 12, (1959),
623-727; \emph{Part II, }Comm. Pure Applied Math. vol. 17, Issue 1, (1964), 35--92.

\bibitem {BLUbook}A. Bonfiglioli, E. Lanconelli, F. Uguzzoni: \emph{Stratified
Lie groups and potential theory for their sub-Laplacians. }Springer Monographs
in Mathematics. Springer, Berlin, 2007.

\bibitem {bb1}M. Bramanti, L. Brandolini. $L^{p}$\emph{-estimates for
uniformly hypoelliptic operators with discontinuous coefficients on
homogeneous groups.} Rend. Sem. Mat. dell'Univ. e del Politec. di Torino, Vol.
58, 4 (2000), 389--433.

\bibitem {BB2}M. Bramanti, L. Brandolini: $L^{p}$\emph{-estimates for
nonvariational hypoelliptic operators with VMO coefficients. }Trans. Amer.
Math. Soc. 352 (2000), no. 2, 781--822.

\bibitem {Bramanti-Cerutti}M. Bramanti, M. C. Cerutti: $W_{p}^{1,2}%
$\emph{-solvability for the Cauchy-Dirichlet problem for parabolic equations
with VMO coefficients.} Comm. Partial Differential Equations 18 (1993), no.
9-10, 1735--1763.

\bibitem {BF}M. Bramanti, M. S. Fanciullo: \emph{The local sharp maximal
function and BMO in locally homogeneous spaces}. (2015) Preprint. http://arxiv.org/abs/1511.02384

\bibitem {BZ}M. Bramanti, M. Zhu. \emph{Local real analysis in locally
homogeneous spaces.} Manuscripta Mathematica. 138, 477--528 (2012).

\bibitem {CZ}A. P. Calder\'{o}n, A. Zygmund: \emph{On the existence of certain
singular integrals. }Acta Math. 88, (1952). 85--139.

\bibitem {Chiarenza-Frasca-Longo}F. Chiarenza, M. Frasca, P. Longo: $W^{1,2}%
$\emph{-solvability of the Dirichlet problem for nondivergence elliptic
equations with VMO coefficients. }Trans. Amer. Math. Soc. 336 (1993), no. 2, 841--853.

\bibitem {CRW}R. R. Coifman, R. Rochberg, G. Weiss: \emph{Factorization
theorems for Hardy spaces in several variables. }Ann. of Math. (2) 103 (1976),
no. 3, 611--635.

\bibitem {Coifman-Weiss}R. Coifman; G. Weiss. \emph{ Analyse harmonique
non-commutative sur certains espaces homog\`{e}nes.} Lecture Notes in
Mathematics 242, Springer-Verlag, Berlin-Heidelberg- New York 1971.

\bibitem {fo}G. B. Folland. \emph{Subelliptic estimates and function spaces on
nilpotent Lie groups.} Arkiv f\"{o}r Math. 13, (1975), 161-207.

\bibitem {Jer}D. Jerison: \emph{The Poincar\'{e} inequality for vector fields
satisfying H\"{o}rmander's condition.} Duke Math. J. 53 (1986), no. 2, 503--523.

\bibitem {Ko}J. J. Kohn: \emph{Pseudo-differential operators and
hypoellipticity. }Partial differential equations (Proc. Sympos. Pure Math.,
Vol. XXIII, Univ. California, Berkeley, Calif., 1971), pp. 61--69. Amer. Math.
Soc., Providence, R.I., 1973.

\bibitem {K1}N. V. Krylov: \emph{Parabolic and Elliptic Equations with VMO
coefficients. }Comm. in P.D.E.s, 32 (2007), 453-475.

\bibitem {H}L. H\"{o}rmander: \emph{Hypoelliptic second order differential
equations. }Acta Math. 119 1967 147--171.

\bibitem {NSW}A. Nagel, E. M. Stein, S. Wainger: \emph{Balls and metrics
defined by vector fields. I. Basic properties. }Acta Math. 155 (1985), no.
1-2, 103--147.

\bibitem {Stein}E. M. Stein: \emph{ Harmonic Analysis: Real-Variable Methods,
Orthogonality and Oscillatory Integrals.} Princeton Univ. Press. Princeton,
New Jersey 1993.
\end{thebibliography}
\end{document}